%% file: restart_t.tex
\documentclass[11pt]{elsarticle}
\usepackage{amsfonts,amsmath,amssymb,amsthm,bm}
\usepackage[a4paper]{geometry}
\usepackage{hyperref}
\usepackage[final]{changes}

\newcommand{\Cc}{\mathbb{C}}
\newcommand{\conv}{\mathtt{convergence}}

\newcommand{\Hh}{\underline{H}}
\newcommand{\Hht}{\underline{\widetilde{H}}}
\newcommand{\Ht}{\widetilde{H}}
\newcommand{\Kk}{\mathcal{K}}
\newcommand{\Pe}{\textit{P\!e}}
\renewcommand{\Re}{\mathrm{Re}}
\newcommand{\Rr}{\mathbb{R}}
\newcommand{\Rrnn}{\mathbb{R}^{n\times n}}
\newcommand{\tol}{\mathtt{tol}}

\newtheorem{lemma}{Lemma}

\begin{document}
\begin{frontmatter}

\title{ART: adaptive residual--time restarting for\\
       Krylov subspace matrix exponential evaluations}

\author[kiam]{M.A.~Botchev\corref{cor1}}
\ead{botchev@ya.ru}

\author[cge]{L.A.~Knizhnerman}
\ead{lknizhnerman@gmail.com}

\cortext[cor1]{Corresponding author}

\address[kiam]{Keldysh Institute of Applied Mathematics,
Russian Academy of Sciences, Miusskaya~Sq.~4, 125047 Moscow,
Russia}

\address[cge]{Mathematical Modelling Department, Central Geophysical Expedition, 
123298 Moscow, Russia}

\begin{abstract}
  In this paper a new restarting method for Krylov subspace matrix
  exponential evaluations is proposed.  Since our restarting technique
  essentially employs the residual, some convergence results for the residual
  are given.  We also discuss how the restart length can be adjusted after each
  restart cycle,  which leads to an adaptive restarting procedure.
  Numerical tests are
  presented to compare our restarting with three other restarting methods.
Some of the algorithms described in this paper are a part of the 
Octave/Matlab package expmARPACK available at
\url{http://team.kiam.ru/botchev/expm/}.
\end{abstract}

\end{frontmatter}

\section{Introduction}
Computing actions of the matrix exponential for large usually sparse matrices is an important task  occurring e.g.\ in time integration of ODE (ordinary differential equation) systems, network analysis, Markov chain models and many other problems. Krylov subspace methods appear to be an indispensable tool to carry out this task, especially for nonsymmetric or stiff matrices. Other methods applicable to this task include
\added{polynomial interpolation
methods~\cite{MoretNovati2001,TalEzer1989a,DeRaedt03},
stabilized explicit time stepping methods~\cite{RKC97,Lebedev98,MRAIpap}
and scaled truncated Taylor series approximations~\cite{AlmohyHigham2011}.}

To be computationally efficient, Krylov subspace methods for evaluating actions of the matrix exponential on given vectors often have to rely on so-called restarting techniques~\cite{EiermannErnst06}.
Restarting techniques allow to decrease
memory and computational expenses of the Krylov subspace methods by storing (and working with) only a limited, restricted number of the Krylov subspace basis vectors, in contrast to the nonrestarted methods whose memory and work costs usually
grow with the iteration number.
We emphasize that for general matrix function evaluations no short-term recurrence Krylov subspace
methods exist (such as the conjugate gradient method for solving linear systems). On the other hand,
the design of Krylov subspace restarting techniques for general matrix functions is a more
complicated task than for solving linear systems~\cite{TalEzer2007}.

\added{Different restarting techniques for the Krylov subspace methods to evaluate matrix
  exponential actions have been developed.  In the EXPOKIT package~\cite{EXPOKIT}, the restarting
  is based
  on the division of the time interval $[0,t]$ into smaller intervals to facilitate Krylov subspace
  convergence. Here, $t>0$ is the time at which the solution
  $y(t)=\exp(-tA)v$ to the associated initial value problem (IVP)~\eqref{ivp} has to be computed.
  Authors of~\cite{CelledoniMoret97} propose restarting which employs a residual concept
  (cf.~\eqref{res}):
  the solution of the next restart is an approximate correction with respect to the residual.
  This residual restarting is further developed and tested in~\cite{BGH13}.
  In~\cite[Chapter~3]{PhD_Niehoff} the proposed restarting is based on the
  so-called generalized residual~\cite{HochLubSel97}
  and on the analogy between Krylov subspace methods for solving linear systems and for evaluating
  matrix exponential actions.
  Another restarting is proposed in~\cite{TalEzer2007}: it uses the observation that the Krylov subspace
  methods for evaluating $f(A)v$, with $f(\cdot)$ being a given matrix function and $v$ a given vector,
  can be seen as polynomial methods which interpolate $f$ at the so-called
  Ritz values~\cite{Saad92}.
  Other restarting techniques
  include~\cite{Afanasjew_ea08,PhD_Guettel,Eiermann_ea2011,GuettelFrommerSchweitzer2014,JaweckiAuzingerKoch2018}.}

In this paper we present a new restarting procedure for Krylov subspace methods which is applicable to the matrix exponential.
Our restarting technique has an attractive property, namely, convergence of the restarted Krylov subspace method to the sought after solution $e^{-tA} v$ is guaranteed (though can be slow) for \emph{any} length of the restart.
\added{Since our restarting procedure relies on the behavior of the residual (cf.~\eqref{res}),
we provide some convergence estimates for the residual norm.}
\replaced{We also show how }{Another attractive property of }the restarting we present here
\deleted{is that it }can be extended to an adaptive procedure to choose a good 
length for the next restart. 
Numerical tests are presented showing the potential of our adaptive restarting.
We also discuss how this restarting can be combined with the shift-and-invert (SAI)
rational Krylov subspace approximations.
\added{Another contribution of our paper is that the new restarting is tested
  numerically together with three other restarting methods: the time step restarting
  of EXPOKIT~\cite{EXPOKIT}, the Niehoff--Hochbruck restarting~\cite{PhD_Niehoff} and
  the residual restarting~\cite{CelledoniMoret97,BGH13}.}


The paper is organized as follows.  In Section~\ref{sect:rst}
some preliminaries on the Krylov subspace approximations to
the matrix exponential and their residuals are given,
as well as some results on residual convergence.
Furthermore, in this section the new restarting procedure is discussed and an adaptive
way to choose an optimal Krylov dimension based on this
restarting procedure is introduced.  In this section we also discuss how
the restarting can be generalized to the shift-and-invert (SAI)
Krylov subspace approximations.
In Section~\ref{sect:num} we describe numerical experiments
and present their results.  Finally, the conclusions are drawn in
Section~\ref{sect:fin}.

\section{Restarting procedure}
\label{sect:rst}
\subsection{Preliminaries}
Throughout the paper, unless indicated otherwise,
$\|\cdot\|$ denotes the 2-norm.
We also assume that $A\in\Rrnn$ is a matrix such that 
\begin{equation}
\label{fov}
\Re (x^*Ax)\geqslant 0, \qquad \forall x\in\Cc^n,
\end{equation}
where $\Re(z)$ denotes the real part of a complex number $z$.
This means that initial value problem
\begin{equation}
\label{ivp}
y'(t)=-Ay(t), \qquad y(0)=v,
\end{equation}
with $v\in\Rr^n$ given, is well posed, see
e.g.~\cite{HundsdorferVerwer:book}.
Moreover, \eqref{fov} implies that
\begin{equation}
\label{nexpA}
\|\exp(-tA)\| \leqslant e^{-t\omega}, \qquad t\geqslant 0, 
\end{equation}
where $-\omega=\mu(-A)\leqslant 0$ is the logarithmic 2-norm of the
matrix $-A$~\cite{Dekker-Verwer:1984,HundsdorferVerwer:book}.
Let $\Kk_k(A,v)$ be the Krylov subspace 
$$
\Kk_k(A,v) = \mathrm{span}\{ v,Av,\dots, A^{k-1}v \}.
$$
In practice, the Arnoldi process (or, if $A$ is (possibly skew) Hermitian,
the Lanczos process)
is usually exploited to compute an orthonormal basis $v_1$, \dots, $v_k$
of $\Kk_k(A,v)$ (see for
instance~\cite{GolVanL,Parlett:book,Henk:book,SaadBook2003}):
\begin{equation}
\label{AV=VH}
\begin{gathered}
V_k = \begin{bmatrix} v_1 & \dots & v_k \end{bmatrix}\in\Rr^{n\times k},
\qquad \mathrm{colspan} (V_k) = \Kk_k(A,v),
\\
AV_k=V_{k+1}\Hh_k,
\end{gathered}
\end{equation}
where $\Hh_k\in\Rr^{(k+1)\times k}$ is an upper Hessenberg matrix.
The relation $AV_k=V_{k+1}\Hh_k$ is called the Arnoldi decomposition.
Using the fact that the last row of $\Hh_k$ contains a single
nonzero entry $h_{k+1,k}\geqslant 0$, one can rewrite~\eqref{AV=VH} as
\begin{equation}
\label{Arn}
AV_k=V_{k+1}\Hh_k = V_kH_k + h_{k+1,k}v_{k+1}e_k^T,  
\end{equation}
where $H_k=V_k^TAV_k$ is the matrix $\Hh_k$ with the last
row skipped and $e_k=(0,\dots,0,1)^T\in\Rr^k$.
A Krylov subspace approximation $y_k(t)$ to
$y(t)=\exp(-tA)v$ can be computed
as~\cite{Henk:f(A),DruskinKnizh89,Knizh91,Saad92,HochLub97}
\begin{equation}
\label{yk}
y_k(t) = V_k\exp(-t H_k)(\beta e_1),  
\end{equation}
where $e_1=(1,0,\dots,0)^T\in\Rr^k$ and $\beta=\|v\|$.


A convenient way to monitor the error of the Krylov subspace approximation 
$y_k(t)$ is to check its residual with respect to the ODE in~\eqref{ivp}
\begin{equation}
\label{res}
r_k(t) = -y_k'(t) - Ay_k(t),
\end{equation}
which is readily computable
as~\cite{CelledoniMoret97,DruskinGreenbaumKnizhnerman98,BGH13}
\begin{equation}
\label{rk}
r_k(t) = \beta_k(t) v_{k+1},
\qquad \beta_k(t)=-h_{k+1,k}e_k^T  \exp(-t H_k)(\beta e_1).
\end{equation}
Although this residual notion for the matrix exponential is
known and used in the literature (for example, to obtain upper bounds on the error,
see~\cite[formula~(32)]{DruskinGreenbaumKnizhnerman98} and \cite[Lemma~4.1]{BGH13}),
there are hardly any convergence estimates for $r_k(t)$ available.  
Therefore, since we essentially use this exponential residual in this
paper, we now first give a general a priori convergence
result for its norm.

\input{fabser_mod.tex}

We now discuss our restarting procedure.

\subsection{Restarting algorithm}
The restarting procedure we propose here is based on the observation that, even for 
very small values of $k$, the residual $r_k(s) $ is small
in norm on some interval $s\in [0,\delta]$ where $\delta$ is taken
sufficiently small. This is formulated 
more precisely in the following statement
(which is a simple result given here for completeness).
Let us define function $\varphi_1(z)$ as (see e.g.~\cite{HochbruckOstermann2010})
$$
\varphi_1(z)=(e^z-1)/z.
$$

\begin{figure}
  \centering{\texttt{\begin{minipage}{0.85\linewidth}
           \% Given: $t>0$, $\tol>0$, $\Hh_k$ (cf.~\eqref{AV=VH}-\eqref{yk})\\
           \% determine $n_t$ - number of time steps to monitor the residual\\
           $u := e_1$                              \\
           $n_t:= 100$     \% initial value of $n_t$\\
           $r := 2\tol$                                             \\
           while $r>\tol$                          \\\hspace*{2em}
               $E := \exp( -(t/n_t)H_k )$         \\\hspace*{2em}
               $u := Ee_1$                        \\\hspace*{2em}
               $r := -h_{k+1,k} (e_k^T u)$          \\\hspace*{2em}
               if $|r|\leqslant \tol$               \\\hspace*{4em} 
                  $u := e_1$                        \\\hspace*{4em}
                  break the while loop              \\\hspace*{2em}
               end                                  \\\hspace*{2em}
               $n_t:= 2n_t$                         \\
           end                                      \\
           \% compute residual for intermediate time points \\
           \% until it exceeds tolerance                    \\
           for $i=1,\dots,n_t$                       \\\hspace*{2em}
               $u_0 := u$                            \\\hspace*{2em}
               $u := Eu$                             \\\hspace*{2em}
               $r := -h_{k+1,k} (e_k^Tu)$            \\\hspace*{2em}
               if $|r|> \tol$                        \\\hspace*{4em} 
                   $\delta_k := \frac{i-1}{n_t} t$    \\\hspace*{4em}
                   $u := u_0$                        \\\hspace*{4em} 
                   break the for loop                \\\hspace*{2em} 
               end                                   \\ 
           end
  \end{minipage}}}
  \caption{
  This algorithm determines $\delta_k$ such that $\|r_k(s)\|\leqslant\beta\tol$ for
  $s\in[0,\delta_k]$ \added{and $u\in\Rr^k$ such that $y_k(\delta_k)=V_ku$.}
  The initial value of $n_t$ can be changed if desired.}
\label{alg_delta}
\end{figure}

\begin{lemma}
\label{l1}
Let $y_k$ be the Krylov approximation~\eqref{yk} to $y(t)=\exp(-tA)v$
and let $k$ be fixed.
Then for the residual $r_k(t)$ defined by~\eqref{rk} holds
$$
\|r_k(t)\| \leqslant t h_{k+1,k} \|H_k\|\beta\varphi_1(-t\omega_k),
$$
where $-\omega_k=\mu(-H_k)$ is the logarithmic norm of $-H_k$.
Hence, for any $k$ and $\varepsilon>0$  there exists a $\delta_k>0$ such that
$$
\|r_k(s)\| \leqslant \varepsilon, \qquad \forall s \in [0,\delta_k].
$$
\end{lemma}

\begin{proof}
It follows from~\eqref{rk} that
$$
\|r_k(t)\| = |\beta_k(t)| = h_{k+1,k} \cdot | e_k^T u(t)|, 
$$
where $u(t)=\exp(-t H_k)(\beta e_1)$ solves the initial
value problem
$$
u'(t) = -H_ku(t), \quad u(0)= \beta e_1.
$$
Then we have
$$
u(t)-u(0) = (\exp(-t H_k) - I )u(0) = -tH_k\varphi_1(-tH_k)u(0).
$$
Since $H_k=V_k^TAV_k$ is a Galerkin projection of $A$,
from~\eqref{fov} it follows that
\begin{equation}
\label{Ritz}
\lambda_{\min}(\frac12(A+A^T))\leqslant   \lambda_{\min}(\frac12(H_k+H_k^T)),
\qquad
\lambda_{\max}(\frac12(H_k+H_k^T))\leqslant \lambda_{\max}(\frac12(A+A^T)),
\end{equation}
where $\lambda_{\min}(\cdot)$ and $\lambda_{\max}(\cdot)$
denote the minimal and maximal eigenvalues, respectively. 
We further note that
$$
\lambda_{\min}(\frac12(A+A^T))=\omega,
$$
where $\omega$ is the constant defined in~\eqref{nexpA}
(see e.g.~\cite{Dekker-Verwer:1984,HundsdorferVerwer:book})
and denote $\omega_k=\lambda_{\min}(\frac12(H_k+H_k^T))$.
Hence, $H_k$ inherits the property~\eqref{nexpA} of $A$, so that
\begin{equation}
\label{nexpH}
\|\exp(-tH_k)\| \leqslant e^{-t\omega_k}, \qquad t\geqslant 0, 
\end{equation}
with $-\omega_k=\mu(-H_k)\leqslant-\omega$.
Using the estimate in the proof of Lemma~2.4
in~\cite{HochbruckOstermann2010}, one can see that
\eqref{nexpA} implies
\begin{equation}
\label{nPhi}
\|\varphi_1(-tA)\|\leqslant \varphi_1(-t\omega),
\end{equation}
so that a similar relation holds for $\|\varphi_1(-tH_k)\|$.
Hence,
\begin{equation}
\label{no_shrp}  
\begin{aligned}
|e_k^Tu(t)|=|e_k^T(u(t)-u(0))|& \leqslant\|u(t)-u(0)\|
\\
& \leqslant t\|H_k\|\,\|\varphi_1(-tH_k)\|\,\|u(0)\|
\leqslant t\|H_k\|\beta\varphi_1(-t\omega_k).
\end{aligned}
\end{equation}
Then the last part of the Lemma statement follows from the observation
that $0<\varphi_1(z)\leqslant 1$ for $z\leqslant 0$.
\end{proof}

\begin{figure}
  \centering{\texttt{\begin{minipage}{0.8\linewidth}
\% Given: $A\in\Rrnn$, $v\in\Rr^n$, $t>0$, $k_{\max}$ and $\tol>0$ \\
$\conv := \mathtt{false}$                          \\
while not($\conv$)                                 \\\hspace*{2em}
   $\beta:= \|v\|$, $v_1 := \frac1\beta v$         \\\hspace*{2em}
   for $k=1,\dots,k_{\max}$                         \\\hspace*{4em}
       $w := A v_k$                                \\\hspace*{4em}
       for $i=1,\dots,k$                           \\\hspace*{6em}
           $h_{i,k} := w^T v_i$                     \\\hspace*{6em}
           $w      := w - h_{i,k}v_i$               \\\hspace*{4em}
       end                                         \\\hspace*{4em}
       $h_{k+1,k} := \|w\|$                         \\\hspace*{4em}
       $s  := (t/6) (1, 2, 3, 4, 5, 6)^T\in\Rr^6$                   \\\hspace*{4em}
       for $q=1,\dots,6$                                            \\\hspace*{6em}
           $u   := \exp(-s_qH_k)e_1$                                \\\hspace*{6em}
           $r(s_q) := -h_{k+1,k}(e_k^Tu)$                              \\\hspace*{4em}
       end                                                          \\\hspace*{4em}
       $\mathtt{rnorm} := \max\{r(s_1),\dots,r(s_6)\}$                    \\\hspace*{4em}
       if $\mathtt{rnorm}\leqslant\tol$                             \\\hspace*{6em}
           $\conv := \mathtt{true}$                                 \\\hspace*{6em}
           break                                                    \\\hspace*{4em}
       elseif $k=k_{\max}$                                           \\\hspace*{6em}
           \% -------- restart after $k_{\max}$ steps                \\\hspace*{6em}
           carry out Algorithm in Figure~\ref{alg_delta}:           \\\hspace*{6em}
           find $\delta_k$ and $u=\exp(-\delta_kH_k)e_1$            \\\hspace*{6em}
           $v := V_k(\beta u)$                                     \\\hspace*{6em}
           $t := t - \delta_k$                                     \\\hspace*{6em}
           break                                                   \\\hspace*{4em}
       end                                                         \\\hspace*{4em} 
       $v_{k+1} := \frac1{h_{k+1,k}}w$                              \\\hspace*{2em} 
   end                                                            \\
end                                                               \\
$y_k := V_k(\beta u)$
\end{minipage}}}
  \caption{Description of the RT (residual--time) restarting algorithm.
  The algorithm computes Krylov subspace approximation $y_k(t)\approx\exp(-tA)v$ such that
  for its residual $r_k(t)$ holds $\|r_k(s)\|\leqslant\tol$ for $s=t$ and several
  representative points $s\in (0,t)$.}
\label{f:rst_alg}
\end{figure}

It is not difficult to see that we can not expect the estimate
of Lemma~\ref{l1} to be sharp.  Indeed, the sharpness is lost 
in the estimates~\eqref{no_shrp} and this is confirmed by 
numerical experiments given below in this section (see 
also Figure~\ref{f:est}).
We now describe how our restarting procedure works
and give some numerical illustrations.

Let $\tol$ be the given residual tolerance, i.e., we need
to compute a Krylov subspace approximation $\tilde{y}(t)$ 
to $\exp(-tA)v$ whose residual $\tilde{r}$ satisfies
$\|\tilde{r}(s)\|\leqslant\tol$ for all $s\in [0,t]$.
Furthermore, let $k_{\max}$ be the largest possible Krylov 
dimension, so that the costs for computing $V_{k+1}$, $\underline{H}_k$ and
storing $V_{k+1}$ are unacceptably high for $k>k_{\max}$.
Denote by $\delta_k$ the length  of the 
interval $[0,\delta_k]\ni s$ such that 
$\|r_k(s)\|\leqslant\tol$.
We carry out (at most) $k_{\max}$ steps of the Krylov method, computing on the way
at every step~$k$ the residual norm $\|r_k(s)\|$ for several
values $s\in[0,t]$ (including the value $s=t$, of course).
If at step $k$ the largest of the computed
values $\|r_k(s)\|$ is below the given tolerance~$\tol$
then we stop at this step (in this case no restarting is needed).
Otherwise, if $k_{\max}$ steps are done but the stopping criterion
is not satisfied then we carry out our time restarting procedure.
This means that we first determine the value $\delta_{k_{\max}}$.
Taking into account Lemma~\ref{l1}, this is not a difficult task
which is carried out by tracking the values of $\|r_k(s)\|$
for increasing $s$ (this procedure is outlined in detail in
Figure~\ref{alg_delta}).
Then we set
$v:=y_{k_{\max}}(\delta_{k_{\max}})$ and solve the IVP~\eqref{ivp}
on a shorter time interval $[0,t-\delta_{k_{\max}}]$.
Again, when solving this IVP on the shortened time interval,
we can apply the same restarting procedure.
We call this restarting method RT restarting (residual--time
restarting) to emphasize its essential dependence on the time behavior
of the residual function $r_k(s)$.  The RT restarting algorithm
is outlined in Figure~\ref{f:rst_alg} and 
schematically illustrated in Figure~\ref{f:rst_ill}.

\begin{figure}
\centering{\includegraphics[width=0.8\textwidth]{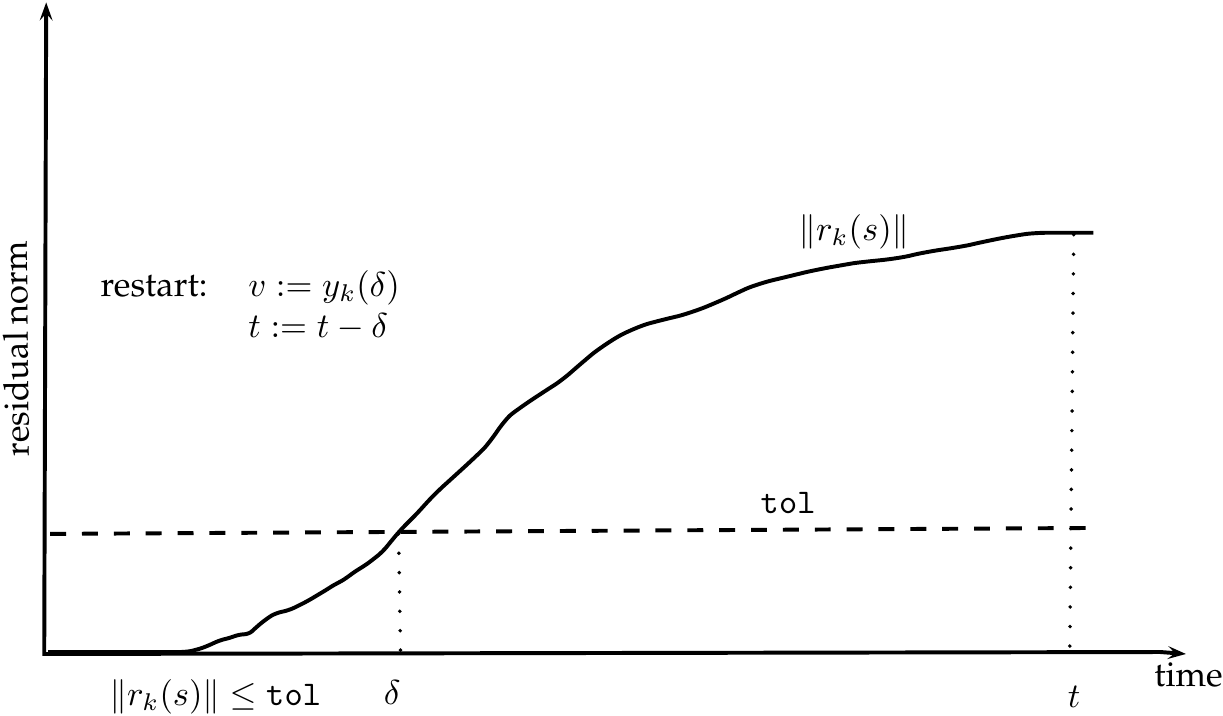}}
\caption{A sketch of the RT restarting procedure}
\label{f:rst_ill}
\end{figure}

It is easy to see that our RT restarting
procedure is guaranteed to converge for \emph{any} restart
length $k_{\max}$ because there always exists a value
$\delta_{k_{\max}}>0$ such that the residual is sufficiently
small in norm on the interval $[0,\delta_{k_{\max}}]$ (see
Lemma~\ref{l1}).

To show how the restarting procedure works, in
Figure~\ref{f:conv_cd} we give convergence plots
of the restarted Arnoldi method applied to the
convection--diffusion test
problem described in Section~\ref{sect:cd}.
The figure also shows the values of $\delta_{k_{\max}}$
plotted against the restart numbers.
In addition, in Figure~\ref{f:rk(s)}
we plot the residual norms $\|r_k(s)\|$ versus
$s\in[0,t]$.  At this point it is instructive
to demonstrate the estimate of Lemma~\ref{l1}.
Therefore, in Figure~\ref{f:est} 
the estimates are plotted which correspond to
the residuals shown in Figure~\ref{f:rk(s)}.
We see that the upper bounds are, as expected,
by no means sharp but they do reflect the time dependence
of the residual norm $\|r_k(s)\|$, $s\in[0,t]$,
qualitatively well.

\begin{figure}
\includegraphics[width=0.48\textwidth]{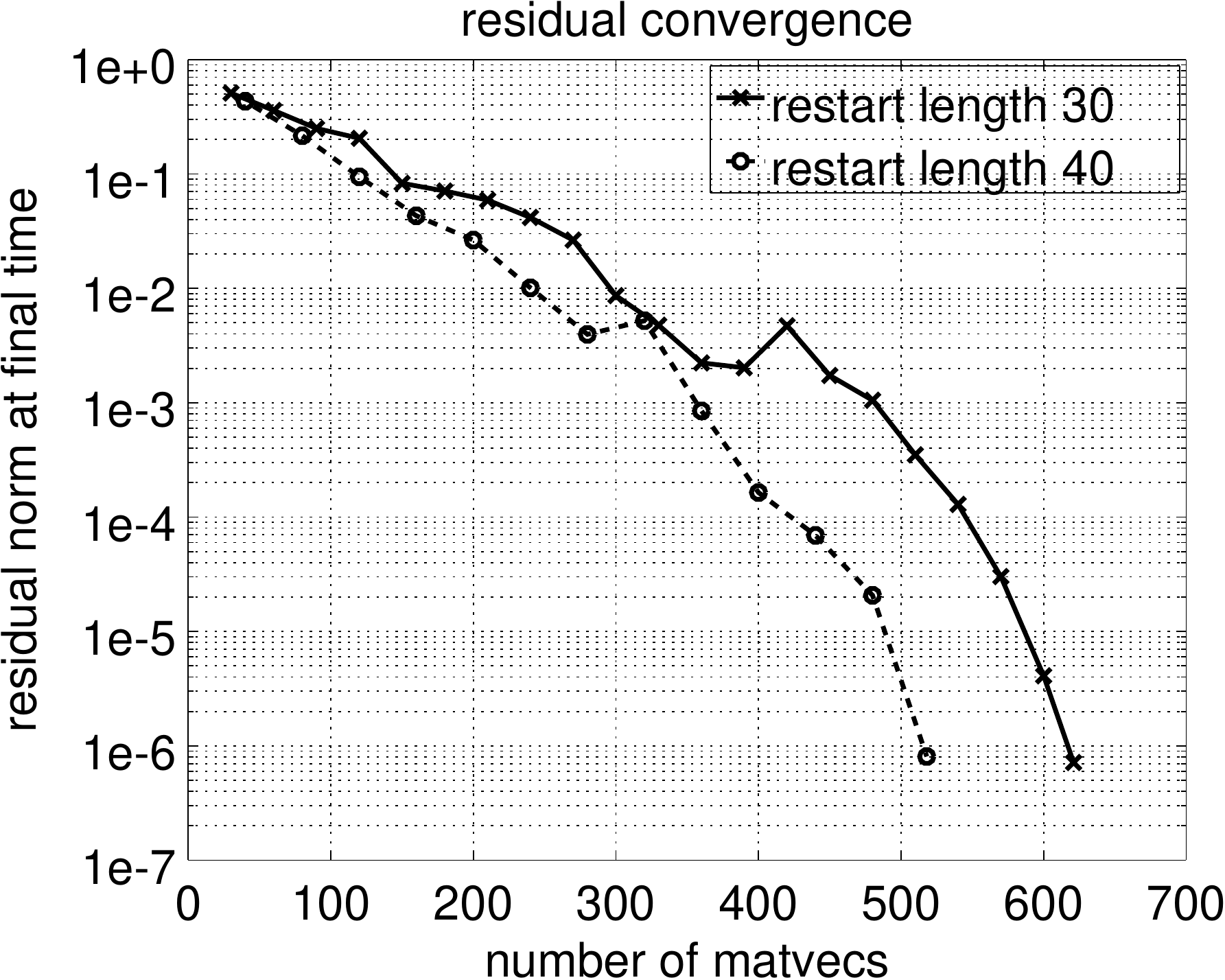}
\includegraphics[width=0.48\textwidth]{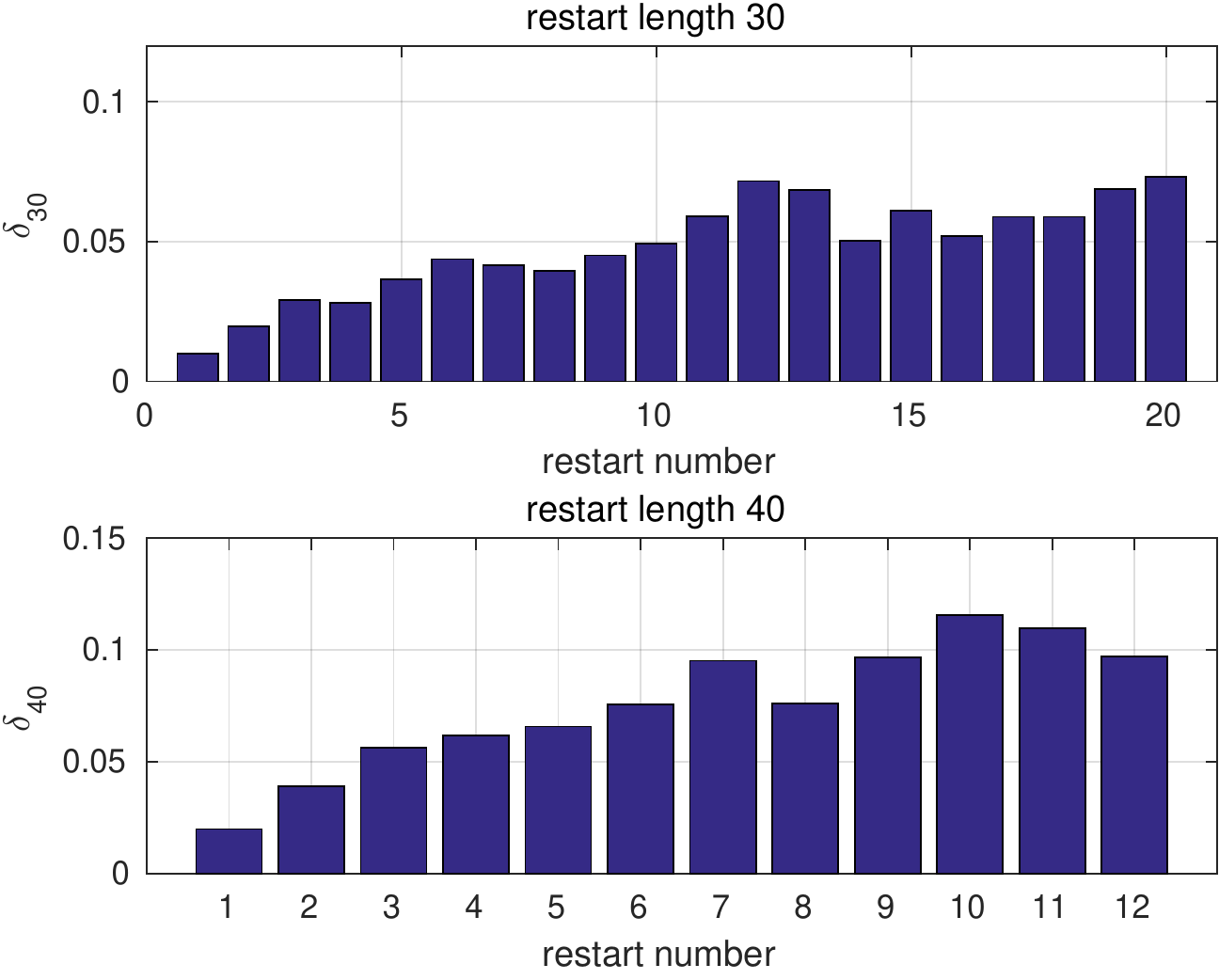}  
\caption{Residual convergence plot (left) and the values of $\delta_{k_{\max}}$
against the restart numbers (right) for the convection--diffusion
test problem (see Section~\ref{sect:cd}).
The matrix $A$ is a discretized convection--diffusion
operator for Peclet number $\Pe=100$, mesh $102\times 102$
($n=100^2$).}
\label{f:conv_cd}  
\end{figure}

\begin{figure}
\includegraphics[width=0.48\textwidth]{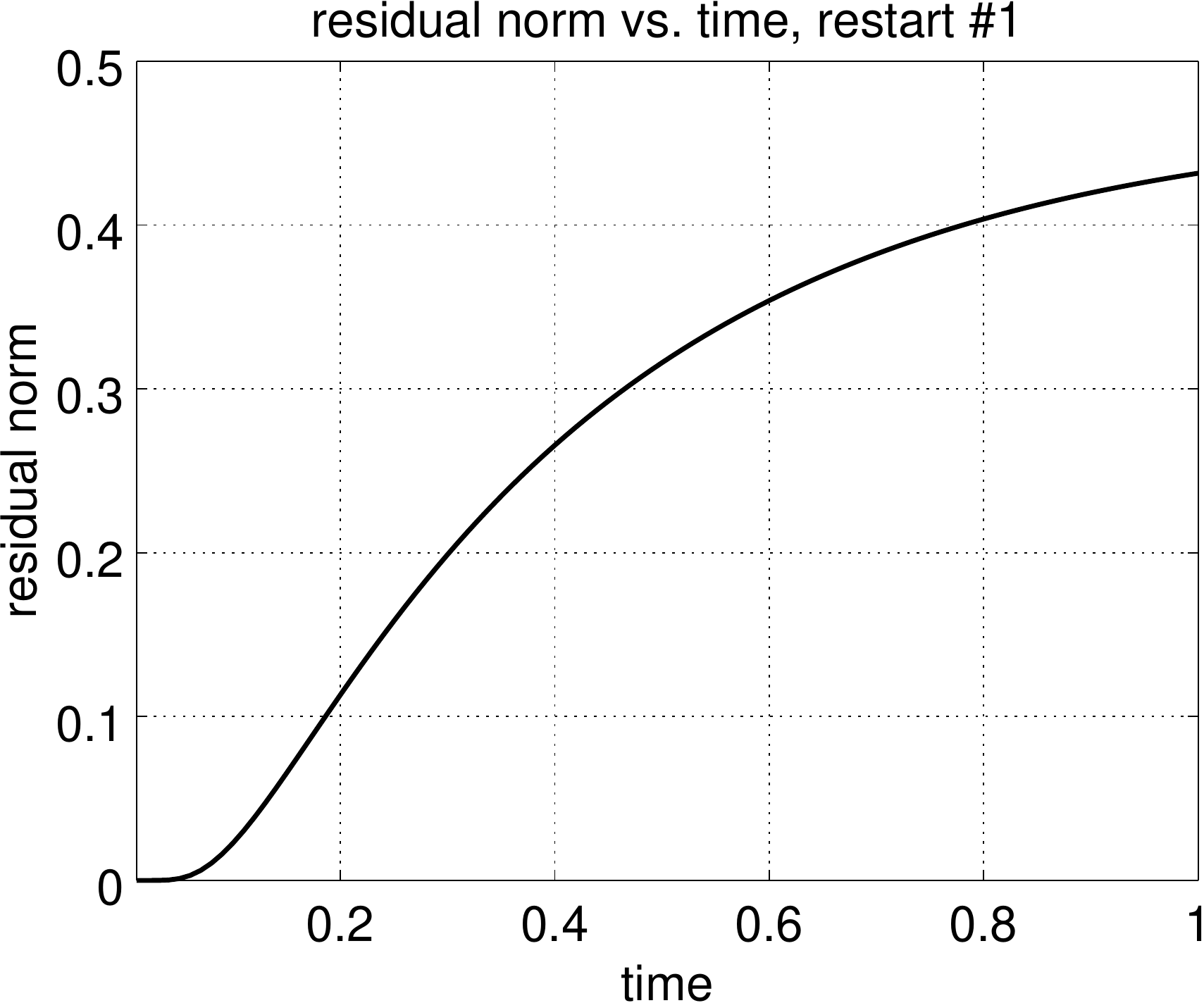}
\includegraphics[width=0.48\textwidth]{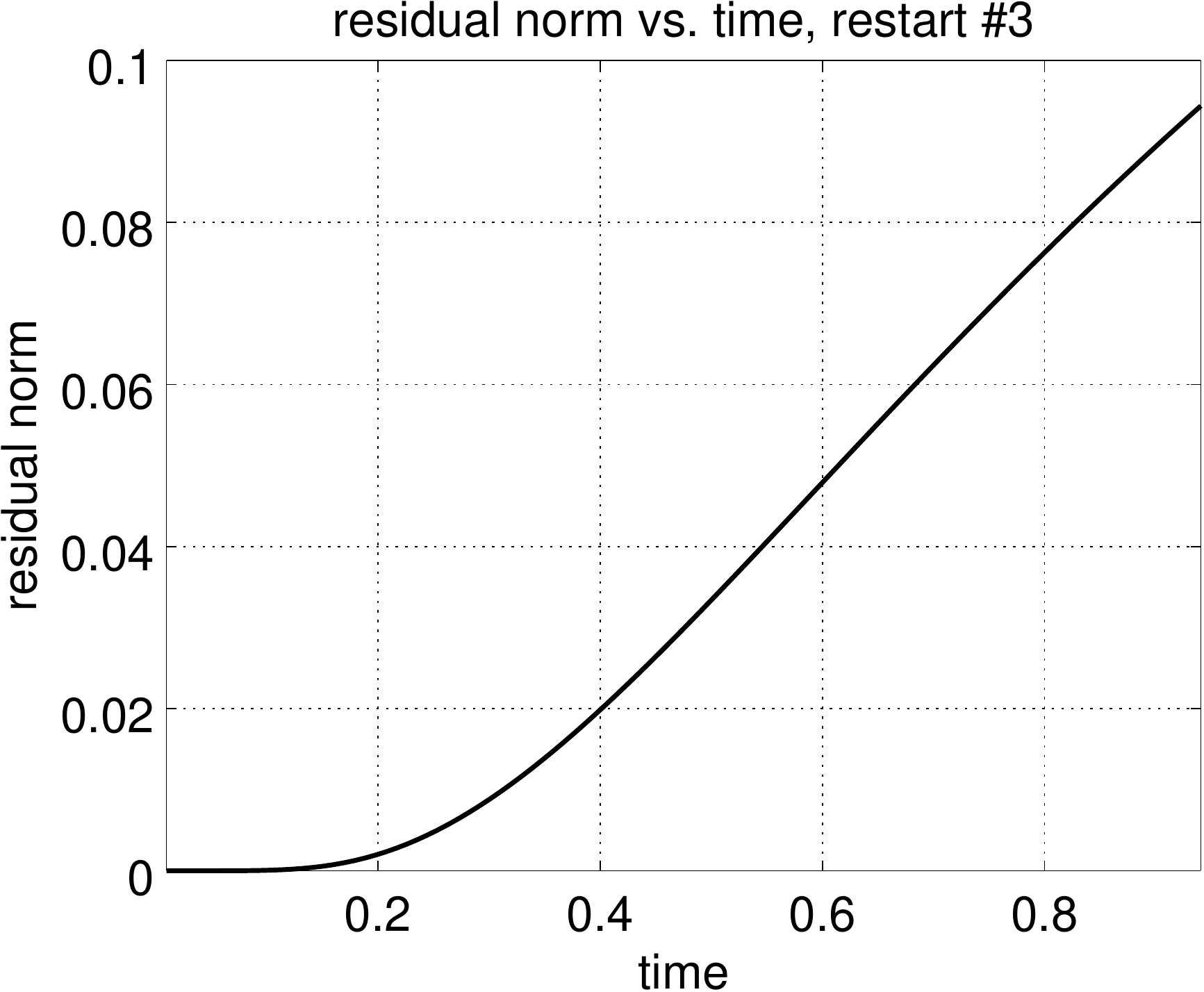}
\caption{Residual norms  $\|r_k(s)\|$ versus $s\in[0,t]$
  at the end of the first (left) and the third (right) restart
  cycles of length~40 for the convection--diffusion
test problem (see Section~\ref{sect:cd}).
The time interval at right is shorter, as it has been
decreased at first two restarts.
The matrix $A$ is a discretized convection--diffusion
operator for Peclet number $\Pe=100$, mesh $102\times 102$
($n=100^2$).}
\label{f:rk(s)}
\end{figure}

\begin{figure}
\includegraphics[width=0.48\textwidth]{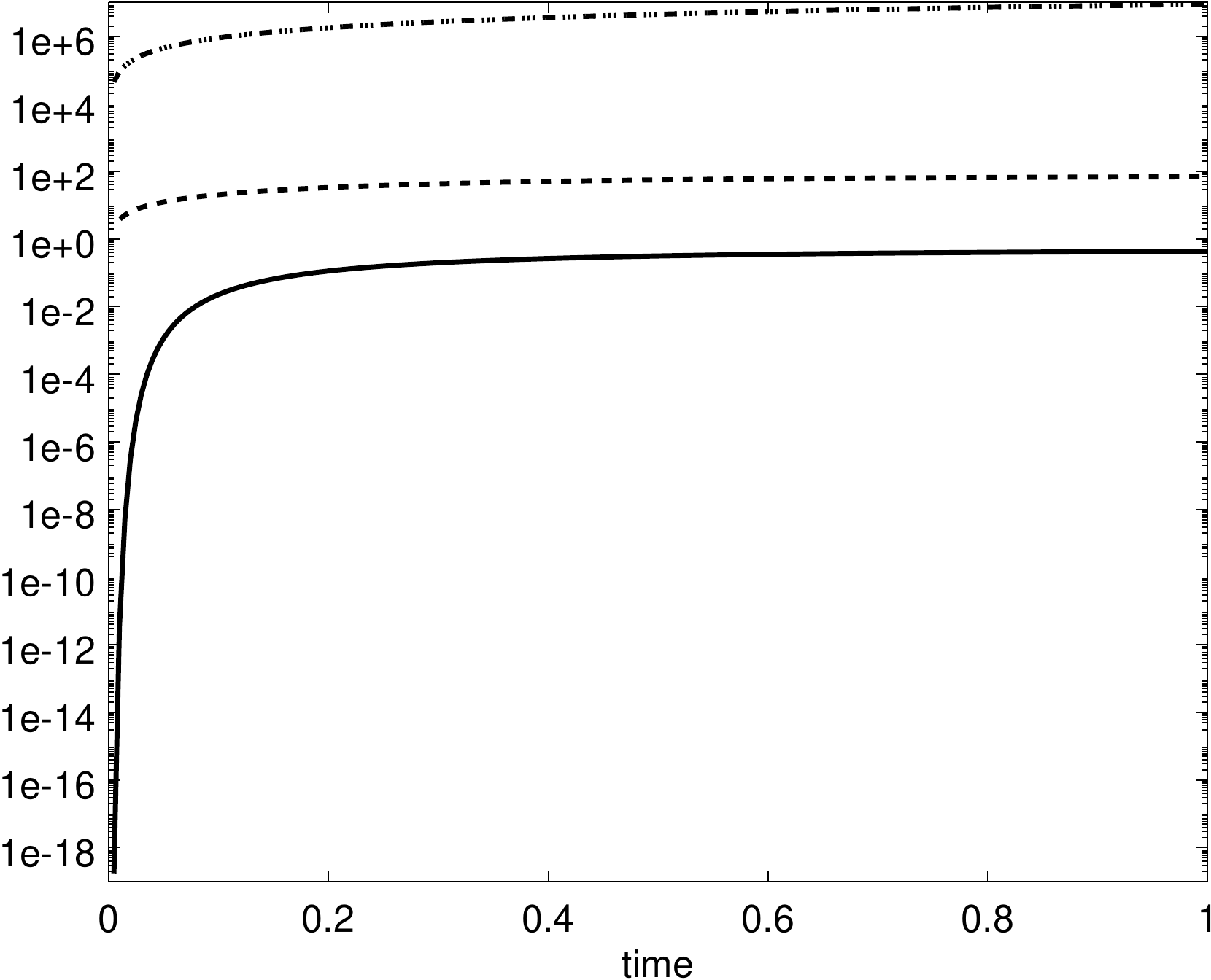}
\includegraphics[width=0.48\textwidth]{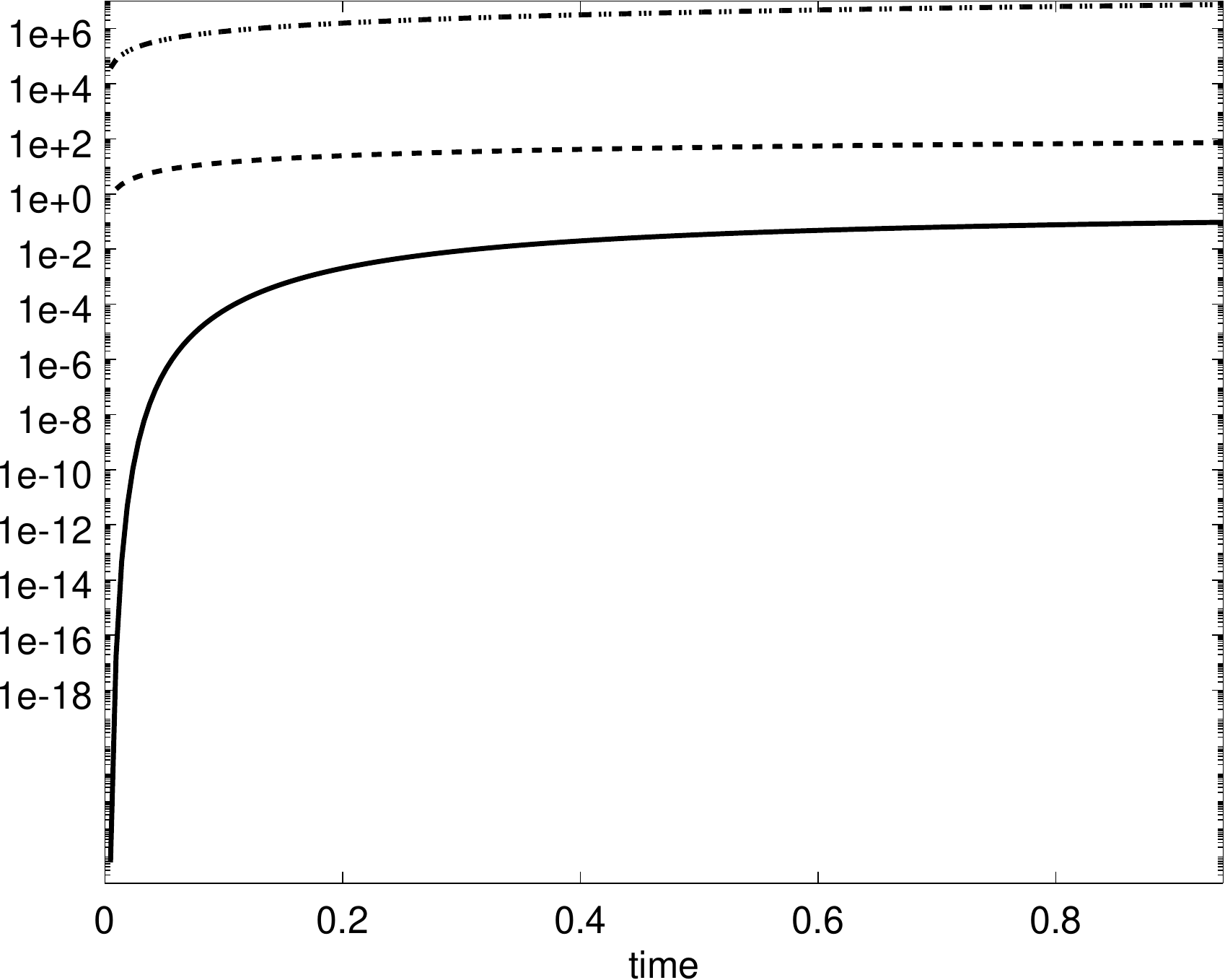}
\caption{Residual norms  $\|r_k(s)\|$ versus $s\in[0,t]$
(left and right, solid lines)
as shown in Figure~\ref{f:rk(s)}, together with their
upper bounds given by Lemma~\ref{l1} and estimates~\eqref{no_shrp}.
Dashed lines: $h_{k+1,k}\|u(s)-u(0)\|$.
Dashed--dotted and dotted lines (coinciding):
$s\beta h_{k+1,k} \|H_k\|\,\|\varphi_1(-sH_k)\|$ and
$s\beta h_{k+1,k} \|H_k\|   |\varphi_1(-s
\omega_k)|$.
}
\label{f:est}
\end{figure}

\subsection{Making the restarting procedure adaptive} 
As can be seen in numerical experiments (see right
plots in Figure~\ref{f:conv_cd}),
the value $\delta_{k_{\max}}$ tends
to remain approximately constant after several first
restarts.
At each restart the time interval
is decreased from $[0,t]$ to $[0,t-\delta_{k_{\max}}]$.
Hence, at the end of each restart we can 
estimate the number of restarts that have yet to be done as
$\approx t/\delta_{k_{\max}}$.

This observation allows to make our time restarting procedure adaptive
as follows.
Denote by $[\cdot]$ the rounding operation to the nearest integer.
While carrying out Krylov steps $k=1,\dots,k_{\max}$, 
we compute for several values $k$ (in our experiments these are 
$[k_{\max}/3]$, $[2k_{\max}/3]$, $[5k_{\max}/6]$ and $k_{\max}$)
the values $\delta_k$ and measured CPU times $t_k^{\mathsf{cpu}}$ spent to
carry out these $k$ steps.
Then the values
$$
\dfrac{t}{\delta_{k}}t_k^{\mathsf{cpu}}, \quad
k=[k_{\max}/3],\; [2k_{\max}/3], \; [5k_{\max}/6],\; k_{\max}
$$
are estimates of the remaining CPU needed to finish the
computations for the restart lengths  
$[k_{\max}/3]$, $[2k_{\max}/3]$, $[5k_{\max}/6]$ and $k_{\max}$, respectively.
Having computed these values during the restart,
we can \replaced{adjust}{ change} the restart length for the next restart to have the smallest expected CPU time.
In the experiments shown below we do so only if the expected gain in CPU time
exceeds 5\%.
This adaptation procedure is then carried out at the end of every restart
and also includes an option for an adaptive increase of the restart length
as follows.
If the current restart length is $\tilde{k}$, with $\tilde{k}<k_{\max}$, and
the CPU time estimations indicate that the restart length $\tilde{k}$ is the
optimal one then the new restart length is set to $\min\{\tilde{k}+5,k_{\max}\}$.
We call this adaptive restarting ART: adaptive residual--time restarting.
In Section~\ref{sect:num} we test the ART restarting numerically and discuss it further.

\subsection{RT restarting for the SAI Krylov subspace method}
\label{sect:sai}
Shift-and-invert (SAI) Krylov subspace methods for computing actions
of the matrix exponential are rational Krylov subspace methods
designed to have a much faster convergence (in terms of the
Krylov subspace dimension) than in regular polynomial
Krylov subspace methods~\cite{MoretNovati04,EshofHochbruck06}.
In SAI methods the Krylov subspace is built up for the matrix
$(I+\gamma A)^{-1}$, $\gamma>0$, rather than for $A$, and thus the price to pay
for the faster convergence is solution, at each Krylov step,
of linear systems with the matrix $I+\gamma A$.
If these linear systems are solved by an iterative method,
there are efficient strategies to save computational work by relaxing the
tolerance to which the systems are solved~\cite{EshofHochbruck06}.
The shift $\gamma$ can be chosen depending on the required tolerance
$\tol$, and for the tolerances of order $10^{-5}$ to $10^{-7}$ a good
value of $\gamma$ is $t/10$ where $t$ is the time interval
length~\cite{EshofHochbruck06}.
This is the value we use in all our experiments.
An attractive property of the SAI Krylov methods is their
often observed space--mesh independent convergence, which
can be proven for the discretizations of parabolic PDEs with a
numerical range close to the positive real
axis~\cite{EshofHochbruck06,GoecklerGrimm2014}.

We now describe how the RT restarting strategy described above can
be applied within the SAI Krylov subspace methods.
The regular Arnoldi decomposition~\eqref{AV=VH},\eqref{Arn}
still holds for the SAI methods, namely
\begin{equation}
\label{Arn_sai}
(I+\gamma A)^{-1}V_k = V_{k+1}\Hht_k  
= V_k\Ht_k + \widetilde{h}_{k+1,k}v_{k+1}e_k^T,
\end{equation}
where $\Ht_k\in\Rr^{k\times k}$ denotes the first $k$ rows of $\Hht_k$.
The approximation $y_k(t)\approx\exp(-tA)v$ is computed according
to~\eqref{yk}, with $H_k$ being the SAI back transformed
projection:
\begin{equation}
\label{saiH}
H_k = \frac1\gamma(\Ht_k^{-1}-I).
\end{equation}
After rewriting the SAI Arnoldi decomposition~\eqref{Arn_sai}
as~\cite[formula~(4.1)]{EshofHochbruck06}
\begin{equation}
\label{Arn_sai2}
AV_k = V_kH_k - \frac{\widetilde{h}_{k+1,k}}{\gamma}(I+\gamma A)v_{k+1}e_k^T\Ht_k^{-1},
\end{equation}
we can see that the residual $r_k(t)=-y'_k(t)-Ay_k(t)$
in the SAI Krylov subspace method reads~\cite{BGH13}
\begin{equation}
\label{rk_sai}
r_k(t)    = \beta_{k} (t) (I+\gamma A)v_{k+1},\qquad
\beta_k(t) = 
\beta\,\frac{\widetilde{h}_{k+1,k}}{\gamma}\, e_k^T\Ht_k^{-1}\exp(-tH_k)e_1.  
\end{equation}
This residual function $r_k(s)$ has quite a different dependence on
$s\in[0,t]$ than the residual in the regular Krylov method~\eqref{rk}:
its convergence $\|r_k(s)\|\rightarrow 0$ with growing $k$ is much
more uniform in $s\in [0,t]$.  This means that we can not expect
that $r_k(s)$ is much smaller in norm for small $s$ than for larger $s$.
Typical dependence of the SAI residual norm $\|r_k(s)\|$ on $s\in [0,t]$
is shown in Figure~\ref{f:rk_sai}.
\begin{figure}
\centering{%
\includegraphics[width=0.48\textwidth]{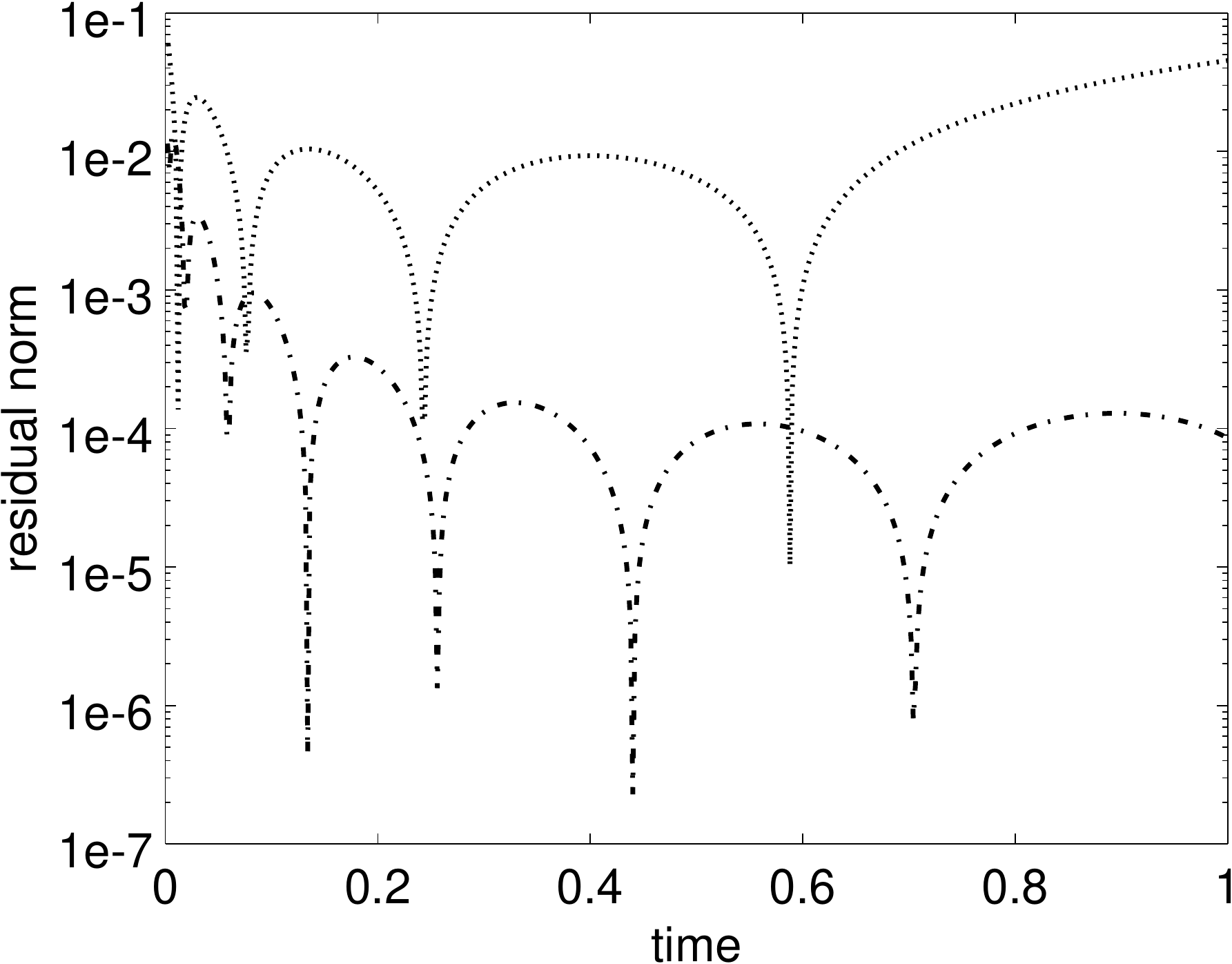}}
\caption{Residual norms  $\|r_k(s)\|$ versus $s\in[0,t]$
of the SAI Krylov subspace method after $k=5$ (dotted line)
and $k=10$ (dash-dotted line) Krylov steps for the convection--diffusion
test problem (see Section~\ref{sect:cd}).
The matrix $A$ is a discretized convection--diffusion
operator for Peclet number $\Pe=100$, mesh $102\times 102$
($n=100^2$).}
\label{f:rk_sai}  
\end{figure}
As we see, there are a few distinct points $s$ on the interval $[0,t]$
where $\|r_k(s)\|$ has values much smaller than the average value
$\frac1t \int_0^t\|r_k(s)\|\, ds$.  Therefore, we can carry out the RT restarting
procedure by assigning $\delta_k$ the value $s$ at which $\|r_k(s)\|$
has the smallest value:
$$
\delta_k := \arg\min_{s\in[0,t]} \|r_k(s)\|.
$$
The rest of the  RT restarting is carried out exactly in the same way
as for regular Krylov subspace methods, i.e., we set
$v:=y_{k_{\max}}(\delta_{k_{\max}})$ and proceed further by solving the IVP~\eqref{ivp}
on a shorter time interval $[0,t-\delta_{k_{\max}}]$.
The drawback of this approach is that the reachable accuracy is restricted
by the value $\|r_k(\delta_k)\|=\min_{s\in[0,t]} \|r_k(s)\|$
(this is the error introduced to the initial data of the restarted IVP
by setting $v:=y_{k_{\max}}(\delta_{k_{\max}})$).
For instance, as can be seen in Figure~\ref{f:rk_sai},
for the restart length $k_{\max}=5$ we have $\delta_k\approx 0.6$ and
the reachable accuracy is $\min_{s\in[0,t]} \|r_k(s)\|\approx 10^{-5}$.
In practical implementations of the method a warning should be given
if such an accuracy drop takes place during the restart (i.e., if
it turns out that $\min_{s\in[0,t]} \|r_k(s)\|$ is larger than
the prescribed tolerance $\tol$).
We further discuss the RT restarted SAI Krylov subspace methods in
Section~\ref{sect:num}.

\section{Numerical experiments}
\label{sect:num}
All the numerical tests are carried out on a Linux PC
with 8~CPUs Intel Xeon~E5504 2.00GHz in Matlab.  No parallelization
is carried out for the regular Arnoldi method.
However, for Arnoldi/SAI methods automatic parallelization of the sparse
LU~factorizations is employed by Matlab. 

\subsection{Convection--diffusion test problem}
\label{sect:cd}
In this test problem the matrix $A$ is a standard five point
central-difference discretization
of the following convection--diffusion operator acting on
$u(x,y)$, defined for $(x,y)\in [0,1]\times [0,1]$ 
and satisfying the homogeneous Dirichlet boundary conditions:
\begin{gather*}
  L[u]=-(D_1u_x)_x-(D_2u_y)_y + \Pe{}\,\left(
  \frac12(v_1u_x + v_2u_y) + \frac12((v_1u)_x + (v_2u)_y) \right)
\\
D_1(x,y)=\begin{cases}
    10^3 \; &(x,y)\in [0.25,0.75]^2, \\
    1       &\text{otherwise},
    \end{cases}\qquad\quad D_2(x,y)=\frac12 D_1(x,y),
\\
v_1(x,y) = x+y,\qquad v_2(x,y)=x-y.
\end{gather*}
where the convective terms (i.e., the terms containing
the first derivatives of $u$) are written in this specific
form to guarantee that the contribution of the convection terms
results in a skew-symmetric matrix~\cite{Krukier79}.
In the experiments, we use matrices $A$ discretized on
mesh $802\times 802$ for the Peclet number $\Pe=200$
and on 
mesh $1202\times 1202$ for the Peclet number $\Pe=300$.
\added{The problem size for these meshes is then
$n=640\,000$ and $n=1\,440\,000$, respectively.}
In both cases $\|\frac12(A+A^T)\|_2\approx 6000$ and
$\|\frac12(A-A^T)\|_2\approx 0.5$, so that the matrices are
weakly non-symmetric.
The initial vector $v$ is set to the values
of the function $\sin(\pi x)\sin(\pi y)$ on the 
finite-difference mesh and then normalized as $v:=v/\|v\|$.
The final time is set to $t=1$ and the tolerance to $\tol=10^{-6}$.

We compare the Krylov subspace method based on the Arnoldi process and our RT restarting
with the following three methods:
\begin{enumerate}
\item
  The \texttt{phiv} function of the EXPOKIT package~\cite{EXPOKIT} based on the
  EXPOKIT time stepping restarting strategy.
\item
  Krylov subspace method based on the Arnoldi process and residual
  restarting \cite[Chapter~3]{PhD_Niehoff}.
\item
  Krylov subspace method based on the Arnoldi process and NH (Niehoff--Hochbruck)
  restarting \cite[Chapter~3]{PhD_Niehoff}.  
\end{enumerate}
We note that our implementations of the Arnoldi method do not include
reorthogonalization of the Krylov basis vectors and we have not
noticed a serious orthogonality loss.
In the tables of this section presenting numerical results,
the last two restarting methods are indicated by ``Res'' and ``NH'',
respectively.
In these tables, the restarting methods presented in this paper are
denoted by ``RT'' (residual--time) and ``ART'' (adaptive RT)
restarting.

\begin{table}
\caption{Results for the convection--diffusion test problem.
  The restart ways are denoted by ``Res'' (residual based),
  ``NH'' (Niehoff--Hochbruck), ``RT'' (residual--time) and
  ``ART'' (adaptive RT).  EXPOKIT uses time--stepping
restarts.}
\label{t:cd}
\centering{\begin{tabular}{llccc}
\hline\hline
method  & restart      &  CPU    & steps& error  \\
        & length, way   &  time, s&     &        \\
\hline
%
\multicolumn{5}{c}{$\Pe=200$, mesh $802\times 802$} \\
%
EXPOKIT & 30           & 57.3   & 800  & \texttt{3.82e-08} \\
Arnoldi & 30, Res      & 67.8   & 316  & \texttt{1.18e-07} \\ 
Arnoldi & 30, NH       & 71.7   & 317  & \texttt{1.02e-08} \\ 
Arnoldi & 30, RT       & 44.6   & 569  & \texttt{2.28e-08} \\ 
Arnoldi & 30, ART      & 41.1   & 572  & \texttt{2.05e-08} \\ 
EXPOKIT & 40           & 63.6   & 756  & \texttt{5.22e-09} \\ 
Arnoldi & 40, Res      & 74.9   & 298  & \texttt{4.40e-08} \\ 
Arnoldi & 40, NH       & 78.1   & 299  & \texttt{9.80e-09} \\ 
Arnoldi & 40, RT       & 45.1   & 505  & \texttt{1.18e-08} \\ 
Arnoldi & 40, ART      & 42.9   & 499  & \texttt{1.27e-08} \\ 
\hline
\multicolumn{5}{c}{$\Pe=300$, mesh $1202\times 1202$} \\
EXPOKIT & 30           & 129.2  & 800  & \texttt{3.25e-08} \\
Arnoldi & 30, Res      & 136.9  & 310  & \texttt{8.34e-08} \\ 
Arnoldi & 30, NH       & 145.0  & 312  & \texttt{9.99e-09} \\ 
Arnoldi & 30, RT       & 90.7   & 539  & \texttt{2.83e-08} \\ 
Arnoldi & 30, ART      & 85.8   & 538  & \texttt{2.55e-08} \\ 
EXPOKIT & 40           & 147.3  & 756  & \texttt{4.09e-09} \\ 
Arnoldi & 40, Res      & 161.1  & 292  & \texttt{2.51e-08} \\ 
Arnoldi & 40, NH       & 154.2  & 293  & \texttt{1.08e-08} \\ 
Arnoldi & 40, RT       & 98.3   & 489  & \texttt{1.26e-08} \\ 
Arnoldi & 40, ART      & 93.7   & 492  & \texttt{1.01e-08} \\ 
\hline
\end{tabular}}
\end{table}

A simple way to restart Krylov subspace evaluations of the matrix exponential
is to split the time interval into smaller intervals (time steps), on which
the method converges within an acceptable number of Krylov steps.
A question arises how this time-stepping restarting approach is compared to
our RT restarting.
To answer this question we include into comparisons the EXPOKIT package,
which exploits such a time-stepping restarting~\cite{EXPOKIT}.

The results of the comparisons are presented in Tables~\ref{t:cd}
and~\ref{t:cd-sai}.
\added{The values reported there in the column ``error''
are relative error norms computed with respect to the
EXPOKIT solution.}
The results in Table~\ref{t:cd} show that our RT and ART restarting
strategies outperform
the other restarting strategies in terms of the CPU time.
The residual based and NH restarting perform worse than the EXPOKIT
restarting in terms of the CPU time.  The reason for this is
probably the sophisticated treatment of the projected problems
in these restarting strategies, which creates an overhead in computational
costs.
Indeed, if, for instance, 5 restarts of length $k_{\max}=30$
are carried out with the NH restarting then the matrix exponential
of a matrix size $150\times 150$ has to be computed.
Additional costs in the residual based restarting is solution
of a small projected ODE system, which, for this restarting, can not be
solved by a matrix exponential evaluation~\cite{BGH13}.
Note also that the matrix-vector products are relatively cheap
for this two-dimensional problems, which makes the other costs
more pronounced.

In Table~\ref{t:cd-sai} the results for Arnoldi/SAI method are
presented.
In this case the method itself converges much quicker than
the regular Arnoldi method, hence, fewer matrix--vector
products are required and the restarting effects are less
visible.  However, we see that all the restarting strategies
perform well in this case.  The drawback of our RT restarting
is that, as discussed in Section~\ref{sect:sai}, it tends to deliver
a less accurate solution  for shorter restart lengths.
We plan to address this problem in future.

\begin{table}
  \caption{Results for the convection--diffusion test problem
    for the Arnoldi/SAI method.
The restart ways are denoted by ``Res'' (residual based),
``NH'' (Niehoff--Hochbruck), ``RT'' (residual--time) and
``ART'' (adaptive RT).  EXPOKIT uses time--stepping
restarts.}
\label{t:cd-sai}
\centering{\begin{tabular}{llccc}
\hline\hline
method  & restart      &  CPU   & steps& error  \\
        & length, way   &  time, s&     &        \\
\hline
\multicolumn{5}{c}{$\Pe=200$, mesh $802\times 802$} \\
Arnoldi/SAI & 30, residual & 46.8   & 16   & \texttt{1.35e-09} \\ 
Arnoldi/SAI & 30, NH       & 46.2   & 14   & \texttt{8.52e-09} \\ 
Arnoldi/SAI & 30, RT       & 44.7   & 14   & \texttt{8.52e-09} \\ 
Arnoldi/SAI & 10, residual & 48.1   & 18   & \texttt{3.90e-07} \\ 
Arnoldi/SAI & 10, NH       & 45.4   & 14   & \texttt{1.24e-08} \\ 
Arnoldi/SAI & 10, RT       & 46.2   & 20   & \texttt{2.50e-07} \\ 
Arnoldi/SAI &  5, residual & 54.5   & 20   & \texttt{1.08e-09} \\ 
Arnoldi/SAI &  5, NH       & 45.4   & 14   & \texttt{1.43e-08} \\ 
Arnoldi/SAI &  5, RT       & 43.8   & 13   & \texttt{4.25e-05} \\ 
\hline
\multicolumn{5}{c}{$\Pe=300$, mesh $1202\times 1202$} \\
Arnoldi/SAI & 10, residual & 159.8  & 17   & \texttt{2.58e-07} \\ 
Arnoldi/SAI & 10, NH       & 155.5  & 14   & \texttt{8.16e-09} \\ 
Arnoldi/SAI & 10, RT       & 154.3  & 14   & \texttt{8.03e-08} \\ 
Arnoldi/SAI &  5, residual & 159.1  & 15   & \texttt{1.75e-07} \\ 
Arnoldi/SAI &  5, NH       & 155.5  & 14   & \texttt{8.31e-09} \\ 
Arnoldi/SAI &  5, RT       & 153.6  & 15   & \texttt{8.14e-06} \\ 
\hline
\multicolumn{5}{c}{$\Pe=300$, mesh $1202\times 1202$, $t=10$}\\
Arnoldi/SAI & 30, Res      & 110.3  & 17   & \texttt{2.94e-09} \\ 
Arnoldi/SAI & 30, NH       & 108.0  & 15   & \texttt{1.25e-08} \\ 
Arnoldi/SAI & 30, RT       & 106.0  & 15   & \texttt{1.25e-08} \\ 
Arnoldi/SAI & 10, Res      & 109.5  & 18   & \texttt{1.42e-06} \\ 
Arnoldi/SAI & 10, NH       & 106.5  & 15   & \texttt{1.57e-08} \\ 
Arnoldi/SAI & 10, RT       & 104.0  & 15   & \texttt{1.57e-05} \\ 
\hline
\end{tabular}}
\end{table}

\subsection{Photonic crystal test problem}
This test problem is a space-discretized system of the three-dimensional
Maxwell equations in a lossless and source-free medium:
\begin{equation}
\label{mxw}
\begin{aligned}
  \frac{\partial\bm{H}}{\partial t} &= -\frac{1}{\mu}\nabla\times\bm{E},
  \\
\frac{\partial\bm{E}}{\partial t} &= \frac{1}{\varepsilon}\nabla\times\bm{H},
\end{aligned}
\end{equation}
where
$\varepsilon$ and $\mu$ are scalar functions of $(x,y,z)$
representing permittivity and permeability, respectively,
and $\bm{H}$ and $\bm{E}$ are vector-valued functions of $(x,y,z,t)$
representing unknown magnetic and electric fields, respectively.
The boundary conditions are homogeneous Dirichlet.
The spatial setup for this test problem is taken from the second
test in~\cite{Kole01}:
in a spatial domain 
$[-6.05, 6.05]\times [-6.05, 6.05]\times [-6.05, 6.05]$
filled with air (relative permittivity $\varepsilon_r=1$) 
there is a dielectric specimen occupying the region
$[-4.55, 4.55]\times [-4.55, 4.55]\times [-4.55, 4.55]$.
The specimen has 27~spherical voids ($\varepsilon_r=1$)
of radius~$1.4$, whose centers have coordinates
$(x_i,y_j,z_k)=(3.03 i, 3.03j, 3.03 k)$,
$i,j,k= -1,0,1$.
The relative permittivity 
in the dielectric specimen is
$\varepsilon_r=5.0$.
The initial values are zero for all the components of both
fields $\bm{H}$ and $\bm{E}$ except for the $x$- and $y$-components
of $\bm{E}$: they are set to nonzero values in the middle of the
spatial domain to represent a light emission.
\added{In addition, the initial value vector $v$ is normalized
  as $v:=v/\|v\|$.}
The standard finite-difference staggered Yee discretization in space
leads to an ODE system of the form~\eqref{ivp}. 
The meshes $40\times 40\times 40$ and $80\times 80\times 80$
used in this test lead to problem size
$n=413\,526$
and
$n=3\,188\,646$, respectively.
The final time is now set to $t=10$ and the tolerance to $\tol=10^{-6}$.

In the second test problem we compare the two solvers which
come out as the best for the first test problem, namely
the \texttt{phiv} solver of EXPOKIT~\cite{EXPOKIT}
and our RT restarting (along with its adaptive version ART).
Our experience reveals that the SAI strategy is inefficient
in this test problem, which can be expected due to 
a strong nonsymmetry of $A$.  This is in contrast
to discretized Maxwell equations with nonreflecting boundary
conditions or in lossy media, where SAI can be
efficient~\cite{VerwerBotchev09,Botchev2016}.

The results are presented in Table~\ref{t:mxw},
where the values reported there in the column ``error''
are relative error norms computed with respect to the
EXPOKIT solution.
As we see, the RT restarting performs well comparably
to EXPOKIT and outperforms it for longer restarts.
We also see that the ART restarting, as expected, indeed helps to
reduce the CPU time, possibly at the cost of the increased total number of Krylov
steps.

%
%
%
%
%
\begin{table}
  \caption{Results for the photonic crystal test problem.
The restart ways are denoted by ``RT'' (residual--time) and
``ART'' (adaptive RT).  EXPOKIT uses time--stepping
restarts.}
\label{t:mxw}
\centering{\begin{tabular}{llccc}
\hline\hline
method  & restart      &  CPU    & steps& error  \\
        & length, way  &  time, s&     &        \\
\hline
\multicolumn{5}{c}{mesh $40\times 40\times 40$} \\
EXPOKIT & 30           & 10.6    & 256  & \texttt{1.02e-06} \\
Arnoldi & 30,  RT      & 11.3    & 231  & \texttt{5.19e-07} \\
Arnoldi & 30, ART      & 11.2    & 229  & \texttt{5.80e-07} \\
EXPOKIT & 70           & 23.3    & 288  & \texttt{2.16e-09} \\
Arnoldi & 70,  RT      & 14.4    & 168  & \texttt{1.17e-07} \\
Arnoldi & 70, ART      & 14.4    & 167  & \texttt{1.57e-07} \\
\hline
\multicolumn{5}{c}{mesh $80\times 80\times 80$} \\
EXPOKIT & 30           & 155.2   & 512  & \texttt{4.37e-07} \\ 
Arnoldi & 30, RT       & 162.0   & 502  & \texttt{3.15e-07} \\ 
Arnoldi & 30, ART      & 160.8   & 488  & \texttt{4.06e-07} \\ 
EXPOKIT & 40           & 170.3   & 420  & \texttt{2.25e-06} \\ 
Arnoldi & 40, RT       & 160.8   & 427  & \texttt{2.77e-07} \\ 
Arnoldi & 40, ART      & 168.9   & 417  & \texttt{9.69e-07} \\ 
EXPOKIT & 50           & 178.9   & 416  & \texttt{2.33e-07} \\ 
Arnoldi & 50, RT       & 169.1   & 383  & \texttt{1.52e-07} \\ 
Arnoldi & 50, ART      & 169.1   & 379  & \texttt{2.31e-07} \\ 
EXPOKIT & 60           & 208.1   & 434  & \texttt{2.43e-08} \\
Arnoldi & 60, RT       & 181.3   & 354  & \texttt{1.06e-07} \\ 
Arnoldi & 60, ART      & 169.0   & 378  & \texttt{1.49e-07} \\ 
EXPOKIT & 70           & 250.5   & 432  & \texttt{2.74e-08} \\ 
Arnoldi & 70, RT       & 190.1   & 338  & \texttt{8.16e-08} \\ 
Arnoldi & 70, ART      & 162.2   & 389  & \texttt{1.06e-07} \\ 
\hline
\end{tabular}}
\end{table}

\section{Conclusions and outlook}
\label{sect:fin}
In this paper a new restarting RT (residual--time) procedure for Krylov subspace matrix exponential
evaluations is proposed, analyzed and tested numerically.  Our restarting is
algorithmically simple as it only relies on evaluation of the readily
available residual~\eqref{rk} and the restarted problem has the same form~\eqref{ivp}
as the original one.
Furthermore, the RT restarting compares favorably to three other restarting techniques,
namely, the time step restarting of EXPOKIT~\cite{EXPOKIT}, the generalized residual restarting
of Niehoff--Hochbruck~\cite{PhD_Niehoff}
and the residual-based restarting~\cite{CelledoniMoret97,BGH13}.
For the rational SAI (shift-and-invert) Krylov subspace approximations
the proposed restarting works well for moderate accuracy requirements.
The RT restarting is also implemented adaptively and the adaptive RT (ART)
Krylov subspace method is available as  a part of the
Octave/Matlab package expmARPACK at \url{http://team.kiam.ru/botchev/expm/}.

Our future research plans include extension of this restarting approach
to nonhomogeneous and nonlinear ODE systems, in combination with
the exponential block Krylov method~\cite{Botchev2013}.
Furthermore, it would be interesting to see how this approach
will work for second order ODE systems, where the matrix cosine and
sine functions are involved.

The authors thank Vladimir Druskin for stimulating discussions.

\bibliography{matfun,my_bib,mxw}
\bibliographystyle{abbrv}

\end{document}

%% file: fabser_mod.tex
\def\square{\hfill\hbox{\vrule\vbox{\hrule\phantom{o}\hrule}\vrule}}
\def\be#1\ee{\begin{equation}#1\end{equation}}
\newcommand{\bea}{\begin{eqnarray}}
\newcommand{\eea}{\end{eqnarray}}
\newcommand{\beas}{\begin{eqnarray*}}
\newcommand{\eeas}{\end{eqnarray*}}
\newcommand{\Int}{\int\limits}
\newcommand{\tJ}{\tilde{J}}
\newcommand{\diag}{\mathop{\rm diag}\nolimits}
\newcommand{\sn}{\mathop{\rm sn}\nolimits}
\newtheorem{theorem}{Theorem}
\newtheorem{remark}{Remark}
\newcommand{\eps}{\varepsilon}
\newcommand{\bars}{\bar{s}}
\newcommand{\barH}{\bar{H}}
\newcommand{\RP}{\mathop{\rm \texttt{\textup{\large \bf RP}}}\nolimits}
\newcommand{\AP}{\mathop{\rm \texttt{\textup{\large \bf AP}}}\nolimits}
\newcommand{\bfN}{\mathbf{N}}
\newcommand{\bfR}{\mathbf{R}}
\newcommand{\bfC}{\mathbf{C}}
\newcommand{\calD}{\mathcal{D}}
\newcommand{\calDc}{\stackrel{3}{\mathcal{D}}}
\newcommand{\calDD}{\stackrel{2}{\mathcal{D}}}
\newcommand{\calDm}{\stackrel{m}{\mathcal{D}}}
\newcommand{\calDtri}{\stackrel{3}{\mathcal{D}}}
\newcommand{\calDmp}{\stackrel{m+1}{\mathcal{D}}}
\newcommand{\Mu}{{\mbox{\textrm{M}}}}
\newcommand{\hh}{\widehat{h}}
\newcommand{\wt}{\widetilde}
\newcommand{\la}{\lambda}
\newcommand{\sjk}{\sum_{\stackrel{\scriptstyle j=1}{j\neq k}}^n}
\newcommand{\lk}{\lambda_k}
\newcommand{\lj}{\lambda_j}
\newcommand{\ydl}{\frac{y_j}{\lambda_k-\lambda_j}}
\newcommand{\odl}{\frac{1}{\lambda_k-\lambda_j}}
\newcommand{\tT}{{\mbox{\it \tiny T}}}
\newcommand{\dsp}{\displaystyle}
\newcommand{\B}{\beta}
\newcommand{\pneq}{{\stackrel{\scriptstyle p=1}{p\neq q}}}
\newcommand{\pner}{{\stackrel{\scriptstyle p=1}{p\neq r}}}
\newcommand{\De}{\Delta}
\newcommand{\Th}{\theta}
\newcommand{\A}{\alpha}
\newcommand{\T}{{\mathcal{T}}}
\newcommand{\s}{{\bf s}}
\newcommand{\Sp}{\mathop{\rm Sp}\nolimits}
\newcommand{\Co}{\mathop{\rm Co}\nolimits}
\newtheorem{proposition}{Proposition}
\newtheorem{remarkR}{Remark}
\newcommand{\dist}{\mathop{\rm dist}\nolimits}
\newtheorem{conjecture}{Conjecture}
\newcommand{\calF}{{\mathcal{F}}}
\newcommand{\spanl}{\mathop{\rm span}\nolimits}
\newcommand{\dn}{\mathop{\rm dn}\nolimits}
\newcommand{\bfCext}{\overline{\mathbf{C}}}
\newcommand{\calK}{\mathcal{K}}
\newcommand{\La}{\left\langle}
\newcommand{\Ra}{\right\rangle}
\renewcommand{\Re}{\mathop{\rm Re}\nolimits}


\subsection{An estimate on the residual in terms of the Faber series}

Faber series as a means to investigate convergence of Arnoldi method are introduced in \cite{Knizh91}; see also \cite{BR09}.
Let $\Phi_j$ be the Faber polynomials~\cite{Suetin} on the compact $W(A)$, and
consider the Faber series decomposition
\be \label{app3}
f(z;t) = e^{-tz} = \sum_{j=0}^{+\infty}f_j(t)\Phi_j(z), 
\qquad z\in W(A), \quad t\geqslant 0,
\ee
where $t$ is considered as a parameter.

\begin{proposition}
The residual of the Arnoldi method $r_k(t)$ defined by~\eqref{res}
satisfies the inequality
\be \label{appr2}
\|r_k(t)\|\leqslant 2h_{k+1,k}\sum_{j=k-1}^{+\infty}|f_j(t)|
\leqslant 2\|A\|\sum_{j=k-1}^{+\infty}|f_j(t)|.
\ee
\end{proposition}

\begin{proof}
Throughout the proof for simplicity we omit the index $\cdot_k$
in $H_k$ and $V_k$.
The Faber series converges superexponentially in~$j$ and its coefficients
$f_j(t)$ are smooth in $t$~\cite[Chapter~3]{Suetin}.
This enables us to differentiate series (\ref{app3}) in~$t$.

Decomposition (\ref{app3}) induces the decomposition of the approximant
\be \label{app6}
y_k=Vf(H;t)e_1=V\sum_{j=0}^{+\infty} f_j(t)\Phi_j(H)e_1.
\ee

Evidently,
\[
zf(z;t) + \frac{\partial f(z;t)}{\partial t} = 0,
\]
whence
\be \label{app1}
0 = z\sum_{j=0}^{+\infty}f_j(t)\Phi_j(z)+\sum_{j=0}^{+\infty}f_j^\prime(t)\Phi_j(z)
= \sum_{j=0}^{+\infty}[f_j(t)z+f_j^\prime(t)]\Phi_j(z).
\ee
Exploiting (\ref{res}), (\ref{app6}), (\ref{AV=VH}), (\ref{app1}) with $H$ substituted for $z$, and the equality $\deg\Phi_j=j$, derive
\beas
-r_k(t)
=AV\sum_{j=0}^{+\infty}f_j(t)\Phi_j(H)e_1
+V\sum_{j=0}^{+\infty}f_j^\prime(t)\Phi_j(H)e_1 \\
=(VH+h_{k+1,k}q_{k+1}e_k^\tT)\sum_{j=0}^{+\infty}f_j(t)\Phi_j(H)e_1
+V\sum_{j=0}^{+\infty}f_j^\prime(t)\Phi_j(H)e_1 \\
=V\sum_{j=0}^{+\infty}[f_j(t)H+f_j^\prime(t)]\Phi_j(H)e_1
+h_{k+1,k}q_{k+1}e_k^\tT\sum_{j=0}^{+\infty}f_j(t)\Phi_j(H)e_1 \\
=h_{k+1,k}q_{k+1}e_k^\tT\sum_{j=0}^{+\infty}f_j(t)\Phi_j(H)e_1
=h_{k+1,k}q_{k+1}e_k^\tT\sum_{j=k-1}^{+\infty}f_j(t)\Phi_j(H)e_1.
\eeas
The bound $\|\Phi_j(H)\|\leqslant2$ (see \cite[theorem~1]{B05}) now implies (\ref{appr2}).
\end{proof}

\begin{remark}
For rendering $|f_j(t)|$ specific see \cite[section~4]{BR09}.
\end{remark}

\begin{remark}
Comparing of our estimate~(\ref{appr2}) with the estimate \cite[theorem~3.2]{BR09} for the error
\[
\|y(t)-y_k(t)\| \leqslant 4\sum_{j=k}^{+\infty}|f_j(t)|,
\]
we see that these two upper bounds differ mainly in coefficients.
This gives one a hope that the error and the residual behave similarly to each other.
We also note that there exist error bounds in terms of the
residual
(see~\cite[formula~(32)]{DruskinGreenbaumKnizhnerman98} and \cite[Lemma~4.1]{BGH13}).
\end{remark}








%% file: restart_t.bbl
\def\ocirc#1{\ifmmode\setbox0=\hbox{$#1$}\dimen0=\ht0 \advance\dimen0
  by1pt\rlap{\hbox to\wd0{\hss\raise\dimen0
  \hbox{\hskip.2em$\scriptscriptstyle\circ$}\hss}}#1\else {\accent"17 #1}\fi}
  \def\cprime{$'$}
\begin{thebibliography}{10}

\bibitem{Afanasjew_ea08}
M.~Afanasjew, M.~Eiermann, O.~G. Ernst, and S.~G\"uttel.
\newblock Implementation of a restarted {K}rylov subspace method for the
  evaluation of matrix functions.
\newblock {\em Linear Algebra Appl.}, 429:2293--2314, 2008.

\bibitem{AlmohyHigham2011}
A.~H. Al-Mohy and N.~J. Higham.
\newblock Computing the action of the matrix exponential, with an application
  to exponential integrators.
\newblock {\em SIAM J. Sci. Comput.}, 33(2):488--511, 2011.
\newblock \url{http://dx.doi.org/10.1137/100788860}.

\bibitem{B05}
B.~Beckermann.
\newblock Image num\'erique, {GMRES} et polyn\^omes de {Faber}.
\newblock {\em C.~R.\ Acad.\ Sci.\ Paris: Ser.~I}, 340(11):855--860, 2005.

\bibitem{BR09}
B.~Beckermann and L.~Reichel.
\newblock Error estimation and evaluation of matrix functions via the {Faber}
  transform.
\newblock {\em SIAM J. Num.\ Anal.}, 47:3849--3883, 2009.

\bibitem{Botchev2013}
M.~A. Botchev.
\newblock A block {Krylov} subspace time-exact solution method for linear
  ordinary differential equation systems.
\newblock {\em Numer.\ Linear Algebra Appl.}, 20(4):557--574, 2013.

\bibitem{Botchev2016}
M.~A. Botchev.
\newblock Krylov subspace exponential time domain solution of {Maxwell's}
  equations in photonic crystal modeling.
\newblock {\em J. Comput.\ Appl.\ Math.}, 293:24--30, 2016.
\newblock \url{http://dx.doi.org/10.1016/j.cam.2015.04.022}.

\bibitem{BGH13}
M.~A. Botchev, V.~Grimm, and M.~Hochbruck.
\newblock Residual, restarting and {Richardson} iteration for the matrix
  exponential.
\newblock {\em SIAM J. Sci.\ Comput.}, 35(3):A1376--A1397, 2013.
\newblock \url{http://dx.doi.org/10.1137/110820191}.

\bibitem{MRAIpap}
M.~A. Botchev, G.~L.~G. Sleijpen, and H.~A. van~der Vorst.
\newblock Stability control for approximate implicit time stepping schemes with
  minimum residual iterations.
\newblock {\em Appl. Numer. Math.}, 31(3):239--253, 1999.

\bibitem{CelledoniMoret97}
E.~Celledoni and I.~Moret.
\newblock A {K}rylov projection method for systems of {ODE}s.
\newblock {\em Appl. Numer. Math.}, 24(2-3):365--378, 1997.

\bibitem{DeRaedt03}
H.~{De Raedt}, K.~Michielsen, J.~S. Kole, and M.~T. Figge.
\newblock One-step finite-difference time-domain algorithm to solve the
  {M}axwell equations.
\newblock {\em Phys.\ Rev.\ E}, 67:056706, 2003.

\bibitem{Dekker-Verwer:1984}
K.~Dekker and J.~G. Verwer.
\newblock {\em Stability of {Runge}--{Kutta} methods for stiff non-linear
  differential equations}.
\newblock North-Holland Elsevier Science Publishers, 1984.

\bibitem{DruskinGreenbaumKnizhnerman98}
V.~L. Druskin, A.~Greenbaum, and L.~A. Knizhnerman.
\newblock Using nonorthogonal {L}anczos vectors in the computation of matrix
  functions.
\newblock {\em SIAM J. Sci. Comput.}, 19(1):38--54, 1998.

\bibitem{DruskinKnizh89}
V.~L. Druskin and L.~A. Knizhnerman.
\newblock Two polynomial methods of calculating functions of symmetric
  matrices.
\newblock {\em U.S.S.R.\ Comput.\ Maths.\ Math.\ Phys.}, 29(6):112--121, 1989.

\bibitem{EiermannErnst06}
M.~Eiermann and O.~G. Ernst.
\newblock A restarted {K}rylov subspace method for the evaluation of matrix
  functions.
\newblock {\em SIAM Journal on Numerical Analysis}, 44:2481--2504, 2006.

\bibitem{Eiermann_ea2011}
M.~Eiermann, O.~G. Ernst, and S.~G{\"u}ttel.
\newblock Deflated restarting for matrix functions.
\newblock {\em SIAM J. Matrix Anal.\ Appl.}, 32(2):621--641, 2011.

\bibitem{GoecklerGrimm2014}
T.~G\"ockler and V.~Grimm.
\newblock Uniform approximation of $\varphi$-functions in exponential
  integrators by a rational {Krylov} subspace method with simple poles.
\newblock {\em SIAM Journal on Matrix Analysis and Applications},
  35(4):1467--1489, 2014.
\newblock \url{http://dx.doi.org/10.1137/140964655}.

\bibitem{GolVanL}
G.~H. Golub and C.~F. {Van Loan}.
\newblock {\em Matrix Computations}.
\newblock The Johns Hopkins University Press, Baltimore and London, third
  edition, 1996.

\bibitem{PhD_Guettel}
S.~G{\"u}ttel.
\newblock {\em Rational {Krylov} Methods for Operator Functions}.
\newblock PhD thesis, Technischen Universit\"at Bergakademie Freiberg, March
  2010.
\newblock \url{www.guettel.com}.

\bibitem{GuettelFrommerSchweitzer2014}
S.~G\"uttel, A.~Frommer, and M.~Schweitzer.
\newblock Efficient and stable {Arnoldi} restarts for matrix functions based on
  quadrature.
\newblock {\em SIAM J. Matrix Anal.\ Appl}, 35(2):661--683, 2014.

\bibitem{HochLub97}
M.~Hochbruck and C.~Lubich.
\newblock On {K}rylov subspace approximations to the matrix exponential
  operator.
\newblock {\em SIAM J. Numer.\ Anal.}, 34(5):1911--1925, Oct. 1997.

\bibitem{HochLubSel97}
M.~Hochbruck, C.~Lubich, and H.~Selhofer.
\newblock Exponential integrators for large systems of differential equations.
\newblock {\em SIAM J. Sci.\ Comput.}, 19(5):1552--1574, 1998.

\bibitem{HochbruckOstermann2010}
M.~Hochbruck and A.~Ostermann.
\newblock Exponential integrators.
\newblock {\em Acta Numer.}, 19:209--286, 2010.

\bibitem{HundsdorferVerwer:book}
W.~Hundsdorfer and J.~G. Verwer.
\newblock {\em Numerical Solution of Time-Dependent
  Advection-Diffusion-Reaction Equations}.
\newblock Springer Verlag, 2003.

\bibitem{JaweckiAuzingerKoch2018}
T.~Jawecki, W.~Auzinger, and O.~Koch.
\newblock Computable strict upper bounds for {Krylov} approximations to a class
  of matrix exponentials and $\phi$-functions.
\newblock {\em arXiv preprint arXiv:1809.03369}, 2018.
\newblock \url{https://arxiv.org/pdf/1809.03369}.

\bibitem{Knizh91}
L.~A. Knizhnerman.
\newblock Calculation of functions of unsymmetric matrices using {A}rnoldi's
  method.
\newblock {\em U.S.S.R.\ Comput.\ Maths.\ Math.\ Phys.}, 31(1):1--9, 1991.

\bibitem{Kole01}
J.~S. Kole, M.~T. Figge, and H.~{De Raedt}.
\newblock Unconditionally stable algorithms to solve the time-dependent
  {Maxwell} equations.
\newblock {\em Phys.\ Rev.\ E}, 64:066705, 2001.

\bibitem{Krukier79}
L.~A. Krukier.
\newblock Implicit difference schemes and an iterative method for solving them
  for a certain class of systems of quasi-linear equations.
\newblock {\em Sov. Math.}, 23(7):43--55, 1979.
\newblock Translation from Izv.\ Vyssh.\ Uchebn.\ Zaved., Mat.\ 1979, No.\
  7(206), 41--52 (1979).

\bibitem{Lebedev98}
V.~I. Lebedev.
\newblock Explicit difference schemes for solving stiff systems of {ODE}s and
  {PDE}s with complex spectrum.
\newblock {\em Russian J. Numer. Anal. Math. Modelling}, 13(2):107--116, 1998.

\bibitem{MoretNovati2001}
I.~Moret and P.~Novati.
\newblock An interpolatory approximation of the matrix exponential based on
  {Faber} polynomials.
\newblock {\em Journal of Computational and Applied Mathematics},
  131(1-2):361--380, 2001.

\bibitem{MoretNovati04}
I.~Moret and P.~Novati.
\newblock {RD} rational approximations of the matrix exponential.
\newblock {\em BIT}, 44:595--615, 2004.

\bibitem{PhD_Niehoff}
J.~Niehoff.
\newblock {\em Projektionsverfahren zur Approximation von Matrixfunktionen mit
  Anwendungen auf die Implementierung exponentieller Integratoren}.
\newblock PhD thesis, Ma\-the\-ma\-tisch\--Na\-tur\-wis\-sen\-schaft\-li\-chen
  Fa\-kul\-t\"at der Heinrich-Heine-Universit\"at D\"usseldorf, December 2006.

\bibitem{Parlett:book}
B.~N. Parlett.
\newblock {\em The Symmetric Eigenvalue Problem}.
\newblock SIAM, 1998.

\bibitem{Saad92}
Y.~Saad.
\newblock Analysis of some {Krylov} subspace approximations to the matrix
  exponential operator.
\newblock {\em SIAM J. Numer. Anal.}, 29(1):209--228, 1992.

\bibitem{SaadBook2003}
Y.~Saad.
\newblock {\em Iterative Methods for Sparse Linear Systems}.
\newblock SIAM, 2d edition, 2003.
\newblock Available from \url{http://www-users.cs.umn.edu/~saad/books.html}.

\bibitem{EXPOKIT}
R.~B. Sidje.
\newblock {\sc Expokit.} {A} software package for computing matrix
  exponentials.
\newblock {\em ACM Trans.\ Math.\ Softw.}, 24(1):130--156, 1998.
\newblock \url{www.maths.uq.edu.au/expokit/}.

\bibitem{RKC97}
B.~P. Sommeijer, L.~F. Shampine, and J.~G. Verwer.
\newblock {RKC}: An explicit solver for parabolic {PDE}s.
\newblock {\em J. Comput.\ Appl.\ Math.}, 88:315--326, 1998.

\bibitem{Suetin}
P.~K. Suetin.
\newblock {\em Series of {Faber} Polynomials}.
\newblock CRC Press, 1998.

\bibitem{TalEzer1989a}
H.~Tal-Ezer.
\newblock Polynomial approximation of functions of matrices and applications.
\newblock {\em Journal of Scientific Computing}, 4(1):25--60, 1989.

\bibitem{TalEzer2007}
H.~Tal-Ezer.
\newblock On restart and error estimation for {K}rylov approximation of
  {$w=f(A)v$}.
\newblock {\em SIAM J. Sci. Comput.}, 29(6):2426--2441, 2007.

\bibitem{EshofHochbruck06}
J.~van~den Eshof and M.~Hochbruck.
\newblock Preconditioning {L}anczos approximations to the matrix exponential.
\newblock {\em SIAM J. Sci.\ Comput.}, 27(4):1438--1457, 2006.

\bibitem{Henk:f(A)}
H.~A. van~der Vorst.
\newblock An iterative solution method for solving {$f(A)x=b$}, using {K}rylov
  subspace information obtained for the symmetric positive definite matrix
  {$A$}.
\newblock {\em J. Comput.\ Appl.\ Math.}, 18:249--263, 1987.

\bibitem{Henk:book}
H.~A. van~der Vorst.
\newblock {\em Iterative {Krylov} methods for large linear systems}.
\newblock Cambridge University Press, 2003.

\bibitem{VerwerBotchev09}
J.~G. Verwer and M.~A. Botchev.
\newblock Unconditionally stable integration of {M}axwell's equations.
\newblock {\em Linear Algebra and its Applications}, 431(3--4):300--317, 2009.

\end{thebibliography}
